\documentclass[12pt,leqno]{article}
\usepackage{amsmath,amsthm,amssymb}
\usepackage{tabularx}
\newcommand{\GF}{{\mathbb F}}

\newcommand{\R}{{\mathbb R}}

\newcommand{\wt}{{\rm wt}}

\DeclareMathOperator{\Harm}{Harm}
\newcommand{\ra}{{\rangle}}
\newcommand{\la}{{\langle}}

\usepackage{color} 

\newtheorem{Thm}{Theorem}[section]
\newtheorem{Lem}[Thm]{Lemma}

\newtheorem{Prop}[Thm]{Proposition}

\theoremstyle{definition}
\newtheorem{Def}[Thm]{Definition}
\newtheorem{Rem}[Thm]{Remark}

\newtheorem{Problem}[Thm]{Problem}

\makeatletter
\@addtoreset{equation}{section}

\makeatother

\begin{document}

\title{The support designs of the triply even binary codes of length $48$}

\author{
Tsuyoshi Miezaki
\thanks{Faculty of Education, University of the Ryukyus, Okinawa  
903-0213, Japan; 
E-mail: miezaki@edu.u-ryukyu.ac.jp
(Corresponding author)
}
and 
Hiroyuki Nakasora
\thanks{Institute for Promotion of Higher Education, Kobe Gakuin University, Kobe
651-2180, Japan;
E-mail: nakasora@ge.kobegakuin.ac.jp
}
}

\date{}

\maketitle

\begin{abstract}

In this paper, we present {examples} of codes
all of whose weight classes support $1$-designs, with duals whose classes include 
two that support $2$-designs.  
We can find {these examples} in the triply even binary codes of length $48$, 
which have been classified by Betsumiya and Munemasa.

\end{abstract}


\vspace{-0.15in}

\quad

\setcounter{section}{+0}
\section{Introduction}

Let $D_{w}$ be the support design of a binary code $C$ for a weight $w$ and 
\begin{align*}
\delta(C)&:=\max\{t\in \mathbb{N}\mid \forall w, 
D_{w} \mbox{ is a } t\mbox{-design}\},\\ 
s(C)&:=\max\{t\in \mathbb{N}\mid \exists w \mbox{ s.t.~} 
D_{w} \mbox{ is a } t\mbox{-design}\}.
\end{align*}

Note that $\delta(C) \leq s(C)$. 
In our previous papers \cite{extremal design H-M-N,extremal design2 M-N}, we considered the following problems. 

\begin{Problem}\label{problem:1}
Find an upper bound of $s(C)$.
\end{Problem}

\begin{Problem}\label{problem:2}
Does the case where $\delta(C) < s(C)$ occur?
\end{Problem}
For Problem \ref{problem:1}, there is no known example of a $6$-design obtained from a code.
For Problem \ref{problem:2}, if $C$ is an extremal Type II code, 
there is no known example of $\delta(C)<s(C)$.
In this paper, we consider the possible occurrence of $\delta(C)<s(C)$ in general.

We can see an example of a case where $\delta(C) < s(C)$ in Dillion and Schatz \cite{Dillion-Schatz}.
An Hadamard difference set over an elementary abelian $2$-group gives a binary $[2^{2m},2m+2]$ code
with the weight enumerator 
\[
1+2^{2m} x^{2^{2m-1}-2^{m-1}}+ (2^{2m+1}-2) x^{2^{2m-1}}+2^{2m}x^{2^{2m-1}+2^{m-1}}+x^{2m}.\]


It has support $2$-designs with weight $2^{2m-1} \pm 2^{m-1}$ and support $3$-design with middle weight $2^{2m-1}$.


A triply even binary code is a linear code in which the weight of
every codeword is divisible by $8$; such codes have previously been 
classified up to length 48 by Betsumiya and Munemasa \cite{Betsumiya-Munemasa 2012, BD}. 
Here we study the support designs of triply even binary codes of length $48$ 
and present the complete list of triply even binary codes of length $48$ 
which have the support $1$-designs obtained from the Assmus--Mattson theorem. 
It is interesting to note that 
some of such codes have a support $2$-design of a different type than the example of Dillion and Schatz

We present the following theorem.

\begin{Thm}\label{thm:main thorem}
There are $46$ triply even binary codes of length $48$ for which every weight class
forms {the} support $1$-design obtained from the Assmus--Mattson theorem, 
and for some $41$ of these, the dual code has the support $2$-designs amongst some,  
but not all, of its weight classes.
\end{Thm}

This paper is organized as follows. 
First, in Section $2$, we give the background material and terminology.
In Section $3$, we review the concept of 
 harmonic weight enumerators, which are used in the proof of the 
main result. 
In Section $4$, 
we present the details of the main result; namely, 
we study the support designs of the triply even binary codes of length $48$. 
Finally, in Section $5$, 
we give some remarks. 


\section{Background material and terminology}\label{sec:terminology}

Let $\GF_q$ be the finite field of $q$ elements. 
A binary linear code $C$ of length $n$ is a subspace of $\GF_2^n$. 
For ${\bf x}:=(x_1,\ldots,x_n)\in C$, the weight of {\bf x} is 
defined as follows:
$\wt({\bf x}):=\sharp\{i\mid x_i\neq 0\}. $
The minimum distance of a code $C$ is 
$\min\{\wt({\bf x})\mid {\bf x}\in C, {\bf x}\neq {\bf 0}\}$. 
A linear code of length $n$, dimension $k$, and 
minimum distance $d$ is called an $[n,k,d]$ code (or $[n,k]$ code for short).

The weight distribution of a code $C$ 
is the sequence $\{A_{i}\mid i=0,1, \dots, n \}$, where $A_{i}$ is the number of codewords of weight $i$. 
The polynomial
$$W_C(x ,y) = \sum^{n}_{i=0} A_{i} x^{n-i} y^{i}$$
is called the weight enumerator of $C$.
The weight enumerator of a code $C$ and its dual $C^{\perp}$ are related. 
The following theorem, due to MacWilliams, is called the MacWilliams identity.
\begin{Thm}[\cite{mac}]\label{thm: macwilliams iden.} 
Let $W_C(x ,y)$ be the weight enumerator of an $[n,k]$ code $C$ over $\GF_{q}$ and let $W_{C^\perp}(x ,y)$ be 
the weight enumerator of the dual code $C^\perp$. Then
$$W_{C^\perp} (x ,y)= q^{-k} W_C(x+(q-1)y,x-y).$$
\end{Thm}

A $t$-$(v,k,{\lambda})$ design (or $t$-design for short) is a pair 
$\mathcal{D}=(X,\mathcal{B})$, where $X$ is a set of points of 
cardinality $v$ and $\mathcal{B}$ a collection of $k$-element subsets
of $X$ called blocks with the property that any $t$ points are 
contained in precisely $\lambda$ blocks.

The support of a nonzero vector ${\bf x}:=(x_{1}, \dots, x_{n})$, $x_{i} \in \GF_{q} = \{ 0,1, \dots, q-1 \}$ is 
the set of indices of its nonzero coordinates: ${\rm supp} ({\bf x}) = \{ i \mid x_{i} \neq 0 \}$\index{$supp (x)$}. 
The support design of a code of length $n$ for a given nonzero weight $w$ is the design 
with $n$ points of coordinate indices and blocks the supports of all codewords of weight $w$.


\section{The harmonic weight enumerators}\label{sec:weight enumerators}

In this section, we review the concept of the harmonic weight enumerators.

A striking generalization of the MacWilliams identity 
was obtained by Bachoc \cite{Bachoc}, 
who gave the concept of  harmonic weight enumerators and 
 a generalization of the MacWilliams identity. 
The harmonic weight enumerators have many applications; 
particularly, the relations between coding theory and 
design theory are reinterpreted and {strengthened} by the harmonic weight 
enumerators \cite {Bachoc,Bannai-Koike-Shinohara-Tagami}. 
For the reader's convenience, we quote 
the definitions and properties of discrete harmonic functions from \cite{Bachoc,Delsarte}. 

Let $\Omega=\{1, 2,\ldots,n\}$ be a finite set (which will be the set of coordinates of the code) and 
let $X$ be the set of its subsets, while, for all $k= 0,1, \ldots, n$, $X_{k}$ is the set of its $k$-subsets.
We denote by $\R X$, $\R X_k$ the real vector spaces spanned by the elements of $X$, $X_{k}$, respectively. 
An element of $\R X_k$ is denoted by
$$f=\sum_{z\in X_k}f(z)z$$
and is identified with the real-valued function on $X_{k}$ given by 
$z \mapsto f(z)$. 

Such an element $f\in \R X_k$ can be extended to an element $\widetilde{f}\in \R X$ by setting, for all $u \in X$,
$$\widetilde{f}(u)=\sum_{z\in X_k, z\subset u}f(z).$$
If an element $g \in \R X$ is equal to some $\widetilde{f}$, for $f \in \R X_{k}$, we say that $g$ has degree $k$. 
The differentiation $\gamma$ is the operator defined by linearity from 
$$\gamma(z) =\sum_{y\in X_{k-1},y\subset z}y$$
for all $z\in X_k$ and for all $k=0,1, \ldots n$, and $\Harm_{k}$ is the kernel of $\gamma$:
$$\Harm_k =\ker(\gamma|_{\R X_k}).$$

\begin{Thm}[{{\cite[Theorem 7]{Delsarte}}}]\label{thm:design}
A set $\mathcal{B} \subset X_{m}$, where $m \leq n$, of blocks is a $t$-design 
if and only if $\sum_{b\in \mathcal{B}}\widetilde{f}(b)=0$ 
for all $f\in \Harm_k$, $1\leq k\leq t$. 
\end{Thm}

In \cite{Bachoc}, the harmonic weight enumerator associated with a binary linear code $C$ was defined as follows. 
\begin{Def}
Let $C$ be a binary code of length $n$ and let $f\in\Harm_{k}$. 
The harmonic weight enumerator associated with $C$ and $f$ is

$$W_{C,f}(x,y)=\sum_{{\bf c}\in C}\widetilde{f}({\bf c})x^{n-\wt({\bf c})}y^{\wt({\bf c})},$$
where ${\bf c}$ is equated with its support.
\end{Def}

Bachoc proved the following MacWilliams-type identity. 
\begin{Thm}[{{\cite[Theorem 2.1]{Bachoc}}}] \label{thm: Bachoc iden.} 
Let $W_{C,f}(x,y)$ be 
the harmonic weight enumerator associated with the code $C$ 
and the harmonic function $f$ of degree $k$. Then 
$$W_{C,f}(x,y)= (xy)^{k} Z_{C,f}(x,y),$$
where $Z_{C,f}$ is a homogeneous polynomial of degree $n-2k$ and satisfies
$$Z_{C^{\perp},f}(x,y)= (-1)^{k} \frac{2^{n/2}}{|C|} Z_{C,f} \left( \frac{x+y}{\sqrt{2}}, \frac{x-y}{\sqrt{2}} \right).$$
\end{Thm}

\section{The support designs of triply even binary codes of length $48$}\label{sec:main}

In this section, we study the support designs of triply even binary codes of length $48$. 


The following theorem is due to Assmus and Mattson \cite{assmus-mattson}. It is one of the 
most important theorems in coding theory and design theory. 
We state it here for the binary case.

\begin{Thm}[\cite{assmus-mattson}] \label{thm:assmus-mattson}
Let $C$ be an $[n,k,d]$ linear code over $\GF_{2}$ and $C^{\bot}$ be the $[n,n-k,d^{\bot}]$ dual code. 
Suppose that for some integer 
$t,0<t<d$, there are at most $d-t$ non-zero weights $w$ in $C^{\bot}$ such that $w \leq n-t$. Then:
\begin{enumerate}
\item[$(1)$] the support design for any weight $u$, $d \leq u \leq n$, in $C$ is a $t$-design;
\item[$(2)$] the support design for any weight $w$, $d^{\bot} \leq w \leq n-t$, 
in $C^{\bot}$ is a $t$-design.
\end{enumerate}
\end{Thm}

If a $t$-design ($t>0$) is obtained from some linear code by this theorem, 
the code is said to be applicable to the Assmus--Mattson theorem.

In \cite{BD, Betsumiya-Munemasa 2012}, triply even binary codes of length 48 are classified {and} 
each code is described {using the notation} $\la \text{Dimension, Code Id, [ Generators ]}\ra$. 
In this paper, we use the {notation} $\la \text{Dimension, [Code Id]} \ra$. 
The following proposition is from a computational search of the
database \cite{BD}. 

\begin{Prop}\label{prop: A 42 and B 4}
If a triply even binary code of length $48$ is applicable to the Assmus--Mattson theorem,
then it is one of the following with $t=1$: 

\begin{enumerate}
\item[$(A)$] 
$\la 7,[144]\ra$, \
$\la8,[129,130,131,132,133]\ra$, \\
$\la9,[59,60,61,62,63,64, 65,66,67,68,69,1109,1712,1714,1716,1960]\ra$, \\
$\la10,[16,17,18,19,20, 21,22,549,550,554,1001,1245,1246,1247]\ra$, \\
$\la11,[6,7,154,520]\ra$, \
$\la12,[3]\ra$, \
$\la13,[1]\ra$. 

\item[$(B)$] 

$\la2,[1]\ra,
\la3,[4]\ra,
\la4,[7]\ra,
\la5,[12]\ra$
\end{enumerate}
\end{Prop}

\begin{proof}
First, for all the codes in the database \cite{BD}, 
we obtained  the weight distribution and minimum weight of the dual. 

Let $C$ be a triply even binary code of length $48$ in $(A)$.
Then $C$ has weight $0, 16, 24, 32, 48$ and $C^{\perp}$ has the minimum weight $4$. 
We apply  Theorem \ref{thm:assmus-mattson} by interchanging $C$ and $C^{\perp}$.
Since there are $4-t=3$ non-zero weights $w$ with $w \leq 48-t$ in $C$, 
we can take $t=1$.

In the case $(B)$, these codes have weight $0, 24, 48$ and 
the dual codes have the minimum weight 2. 
By Theorem \ref{thm:assmus-mattson},
since there are $2-t=1$ non-zero weights $w$ with $w \leq 48-t$, 
we can take $t=1$.

For the other cases except for the cases $(A)$ and $(B)$, 
we have checked numerically that the codes are not applicable to the Assmus--Mattson theorem.

\end{proof}

The following Lemma is easily seen.

\begin{Lem}[{{\cite[Page 3, Proposition 1.4]{CL}}}]\label{lem: divisible 47}

\begin{enumerate}
\item[$(1)$] Let $\mathcal{B}_{1}$ be the block set of a $2$-$(48,k,\lambda_{2})$ design 
for $2 \leq k \leq 46$. 
Then $|\mathcal{B}_{1}|$ is a multiple of $47$.

\item[$(2)$] Let $\mathcal{B}_{2}$ be the block set of a $3$-$(48,6,\lambda_{3})$ design. 
Then $|\mathcal{B}_{2}|$ is divisible by $47 \cdot 23\cdot 4$ 
and $\lambda_{3}$ is divisible by $5$.
\end{enumerate}
\end{Lem}

We now present the main result (the details of Theorem \ref{thm:main thorem}). 

\begin{Thm}\label{thm: 1 and 2-design}
Let $C$ be a triply even binary code of length $48$ in Proposition \ref{prop: A 42 and B 4}.
Let $D_{w}$ and $D_{w}^{\perp}$ be the support $t$-designs of weight $w$ of $C$ and $C^{\perp}$.

\begin{enumerate}
\item[$(1)$] For all $w\neq 0$, $D_{w}$ and $D_{w}^{\perp}$ are $1$-designs.
\item[$(2)$] If $C$ is a code in  Proposition \ref{prop: A 42 and B 4} (A) except for $\la 13,[1]\ra$, 
$D_{6}^{\perp}$ (also $D_{42}^{\perp}$) is a $2$-design  but is not a $3$-design. 
For the other cases, { $D_{w}$ and $D_{w}^{\perp}$ are not $2$-designs.}
\end{enumerate}

\end{Thm}

\begin{proof}
By Proposition \ref{prop: A 42 and B 4}, we have $(1)$.

Let $C$ be a triply even binary code of length $48$ in Proposition \ref{prop: A 42 and B 4} (A).
If $D_{w}$ is the support design for any weight $w$ of $C$, 
we have checked numerically that the number of the blocks of $D_{w}$ is not divisible by $47$. 
Hence, the $D_{w}$ are not $2$-designs by Lemma \ref{lem: divisible 47} (1).

Next, we show that if $C$ is a code according to Proposition \ref{prop: A 42 and B 4} (A) except for $\la 13,[1]\ra$, 
$D_{6}^{\perp}$ (also $D_{42}^{\perp}$) is a $2$-design but is not a $3$-design.
Let $W_{C,f_{i}}(x,y)$ be the harmonic weight enumerator associated with the code $C$ in  Proposition \ref{prop: A 42 and B 4} (A), where $f_i$ denotes a harmonic function of degree $i \geq 1$. 
Since $D_{16}$, $D_{24}$, and $D_{32}$ are $1$-designs, 
for a harmonic function $f_{1}$ of degree $1$, we have
\begin{align*}
W_{C,f_{1}}(x,y) =0 
\end{align*}
by Theorem \ref{thm:design}. 

Let $f_{2}$ be a harmonic function  of degree $2$.
Since $D_{16}$, $D_{24}$, and $D_{32}$ are not $2$-designs and  $D_{32}$ is the complement of $D_{16}$, 
we have 

\begin{align*}
W_{C,f_{2}}(x,y) & =\sum_{{\bf c} \in{C}} \widetilde{f}_{2}(c) x^{48-wt(c)}y^{wt(c)} \\
        & =a x^{32}y^{16}+b x^{24}y^{24}+ax^{16}y^{32} \\
        & =(xy)^{2}(a x^{30}y^{14}+b x^{22}y^{22}+ax^{14}y^{30}) \\
        & =(xy)^{2}Z_{C,f_{2}}(x,y), 
\end{align*}
where $a$ and $b$ are not equal to $0$ by Theorem \ref{thm:design}.

By Theorem \ref{thm: Bachoc iden.}, there exist coefficients $a',b'$ such that
\begin{align*}
& Z_{C^{\perp},f_{2}}(x,y) \\ 
&=(-1)^{2} \frac{2^{24}}{|C|} Z_{C,f_{2}} \left( \frac{x+y}{\sqrt{2}}, \frac{x-y}{\sqrt{2}} \right)\\
& =a' (x+y)^{30}(x-y)^{14}+b' (x+y)^{22}(x-y)^{22}+a'(x+y)^{14}(x-y)^{30}. 
\end{align*}
Since $C^{\perp}$ has minimum weight $4$, the coefficient of $x^{44}$ in $Z_{C^{\perp},f_{2}}$ is equal to $0$. 
Then we have $b'=-2a'$. Hence we have 
\begin{align*}
& W_{C^{\perp},f_{2}}(x,y)\\
&=(xy)^{2} \left( a' (x+y)^{30}(x-y)^{14}-2a' (x+y)^{22}(x-y)^{22}+a'(x+y)^{14}(x-y)^{30} \right).
\end{align*}

By a direct computation, the coefficient of $x^{42}y^{6}$ (also $x^{6}y^{42}$) in $W_{C^{\perp},f_{2}}$ is equal to $0$. 
Hence $D_{6}^{\perp}$ (also $D_{42}^{\perp}$) is a $2$-design.
We have checked numerically that 
the number of  blocks of $D_{6}^{\perp}$ (also $D_{42}^{\perp}$) is not divisible by $47 \cdot 23$. 
Therefore, $D_{6}^{\perp}$ (also $D_{42}^{\perp}$) is not a $3$-design by Lemma \ref{lem: divisible 47} (2). 

In the case $\la13,[1]\ra$, 
this code has the weight enumerator 
\[
x^{48}+759x^{32}y^{16}+6672x^{24}y^{24}+759x^{16}y^{32}+y^{48}.
\]
By Theorem \ref{thm: macwilliams iden.},  
we have $A_{6}^{\perp}=A_{42}^{\perp}=0$, where $A_{i}^{\perp}$ is the numbers  of weight $i$ of the dual code.
Hence there are no blocks of $D_{6}^{\perp}$ and $D_{42}^{\perp}$.


Next, we have checked numerically that 
the number of the blocks of 
$D_{w}^{\perp}$ except for $D_{6}^{\perp}$ and $D_{42}^{\perp}$ 
is not divisible by $47$. 
Therefore, 
$D_{w}^{\perp}$ except for $D_{6}^{\perp}$ and $D_{42}^{\perp}$ are not $2$-designs 
by Lemma \ref{lem: divisible 47} (1).

Let $\widehat{C}$ be a triply even binary code of length $48$ in Proposition \ref{prop: A 42 and B 4} (B).
If $\widehat{D_{w}}$ and $\widehat{D_{w}^{\perp}}$ are the support designs for all weights $w$ of $\widehat{C}$ and $\widehat{C^{\perp}}$, 
we have checked numerically that 
the {numbers} of the blocks of $\widehat{D_{w}}$ and $\widehat{D_{w}^{\perp}}$ {are}  not divisible by $47$. 
Therefore, $\widehat{D_{w}}$ and $\widehat{D_{w}^{\perp}}$ are not $2$-designs 
by Lemma \ref{lem: divisible 47} (1).

The numbers of the blocks are listed on one of the 
author's homepage \cite{miezaki}.
This completes the proof of Theorem \ref{thm: 1 and 2-design}. 
\end{proof}

\begin{Rem}
Three codes, $\la1,[1]\ra$, $\la3,[5]\ra$, and $\la6,[363]\ra$, are not applicable to the Assmus--Mattson theorem, 
but their support designs for all weight are $1$-designs
since these codes are generated by the minimum weight codewords which partition 48 coordinates into equal parts. 
These codes have a transitive automorphism group.
\end{Rem}


\section{$2$-designs from triple even codes of length $48$}

We list $2$-designs obtained from Theorem \ref{thm: 1 and 2-design} in Table \ref{tab: 2-designs}.
{In this section, we give the concluding remarks related to 
$2$-designs of triply even binary codes of length $48$ discussed in Section \ref{sec:main}. }

\begin{Rem}
It is interesting to note that {the dual code of} the first triply even binary code $\la7,[144]\ra$ is called  Miyamoto's moonshine code {\cite{miyamoto}}. 
This triply even binary code has the weight enumerator 
\[x^{48}+3x^{32}y^{16}+120x^{24}y^{24}+3x^{16}y^{32}+y^{48}.
\]
Using Theorem \ref{thm: macwilliams iden.}, we obtained the weight enumerator of the dual code. 
Then we have $A_{6}^{\perp}=189504$.
By Theorem \ref{thm: 1 and 2-design}, $D^{\perp}_{6}$ is a $2$-$(48,6,2520)$ design.

There are five triply even binary codes $[129,130,131,132,133]$ in the case of dimension $8$. 
We have checked by Magma \cite{Magma} that 
these codes give five non-isomorphic $2$-$(48, 6, 1240)$ designs.
Similarly, each triply even binary code in dimensions $9$--$12$ gives a different $2$-design.

In the case of triply even binary code $\la13,[1]\ra$, 
there are no codewords of weight $6$ in the dual code.
\end{Rem}

\begin{Rem}
For the $2$-design $D_6^{\perp}$ obtained from Theorem \ref{thm: 1 and 2-design} in Table \ref{tab: 2-designs}, 
we calculated the automorphism groups of the $2$-designs. 
The {generators and transitivity} of the automorphism groups are 
listed on one of the author's homepage \cite{miezaki}. 
\end{Rem}

\begin{Rem}
We have checked by Magma \cite{Magma} that 
for the $2$-design $D_6^{\perp}$ obtained from Theorem \ref{thm: 1 and 2-design} in Table \ref{tab: 2-designs}, 
the codewords of weight $6$ generate the code $C^\perp$. 

\end{Rem}

\begin{table}[h]
\caption{Support $2$-designs of weight $6$}
\begin{center}
\begin{tabular}{l||l|c} \label{tab: 2-designs}
Dim.~& [Code Id] & $2$-$(v,k,\lambda)$ \\ 
          & Weight distribution $(i, A_{i})$ for $A_{i} \neq 0$ & $\sharp$ of designs \\ \hline \hline 
7  & [144] & $2$-$(48,6,2520)$ \\
   & $(0, 1), (16, 3), (24, 120), (32, 3), (48, 1)$ & 1 \\ \hline

8  & [129,130,131,132,133] & $2$-$(48, 6, 1240)$ \\
   & $(0, 1), (16, 15), (24, 224), (32, 15), (48, 1)$ & 5 \\ \hline

9  & [59,60,61,62,63,64,65,66,67,68,69,1109,1712,1714,1716,1960] & $2$-$(48, 6, 600)$ \\ 
   & $(0, 1), (16, 39), (24, 432), (32, 39), (48, 1)$ & 16 \\ \hline

10 & [16,17,18,19,20,21,22,549,550,554,1001,1245,1246,1247] & $2$-$(48, 6, 280)$ \\ 
   & $(0, 1), (16, 87), (24, 848), (32, 87), (48, 1)$ & 14 \\ \hline

11 & [6,7,154,520] & $2$-$(48, 6, 120)$ \\
   & $(0, 1), (16, 183), (24, 1680), (32, 183), (48, 1)$ & 4 \\ \hline

12 & [3] & $2$-$(48, 6, 40)$ \\
   & $(0, 1), (16, 375), (24, 3344), (32, 375), (48, 1)$ & 1 \\ \hline

13 & [1] & - \\ 
   & $(0, 1), (16, 759), (24, 6672), (32, 759), (48, 1)$ & 0
\end{tabular}
\end{center}
\end{table}

\newpage
\section*{Acknowledgments}
The authors thank Koichi Betsumiya and Akihiro Munemasa for helpful discussions on and computations for this research.
The authors would also like to thank the anonymous
reviewers for their beneficial comments on an earlier version of the manuscript. The first author is supported by JSPS KAKENHI (18K03217). 


\end{document}